\documentclass[12pt]{amsart}
\usepackage[margin=1in]{geometry}
\usepackage{amsmath,amssymb,amsthm,hyperref,graphicx}
\usepackage{mathabx}

\newtheorem{thm}{Theorem}[section]
\newtheorem{prop}[thm]{Proposition}

\theoremstyle{definition}

\newtheorem{defn}[thm]{Definition}

\newtheorem{example}[thm]{Example}

\newtheorem{conj}[thm]{Conjecture}

\newtheorem{lemma}[thm]{Lemma}
\newtheorem*{remark*}{Remark}
\newtheorem{claim}[thm]{Claim}
\newtheorem{fact}[thm]{Fact}
\newtheorem{facts}[thm]{Facts}
\newtheorem{ques}[thm]{Question}

\newcommand{\hatj}{\widehat{j}}
\newcommand{\hati}{\widehat{i}}

\begin{document}

\title[The Alpha Problem \& Line Count Configurations]{The Alpha Problem \& Line Count Configurations}

\author[S. M. Cooper]{Susan M. Cooper}
\address{Department of Mathematics\\
Central Michigan University\\
Mt. Pleasant, MI 48859 USA}
\email{s.cooper@cmich.edu}

\author[S. G. Hartke]{Stephen G. Hartke}
\address{Department of Mathematics\\
University of Nebraska\\
Lincoln, NE 68588-0130 USA}
\email{hartke@math.unl.edu}

\date{March 30, 2014}

\thanks{Acknowledgements: The second author is supported in part by National Science Foundation grant DMS-0914815.  We thank Brian Harbourne and Craig Huneke for introducing us to the problem addressed in this article.  We also thank Jamie Radcliffe and Leo Butler for several illuminating consultations and Matt Young for his helpful suggestions.  The authors also thank the anonymous referee for his or her helpful suggestions and in particular the alternate proof of Claim 2.3.}



\begin{abstract}
Motivated by the work of Chudnovsky and the Eisenbud-Mazur Conjecture on evolutions, Harbourne and Huneke give a series of conjectures that relate symbolic and regular powers of ideals of fat points in projective space $\mathbb P^n$.  The conjectures involve both containment statements and bounds for the initial degree in which there is a non-zero form in an ideal.  Working with initial degrees, we verify two of these conjectures for special line count configurations in projective 2-space over an algebraically closed field of characteristic 0.
\end{abstract}

\maketitle

\section{Introduction}

The extent to which symbolic and regular powers of a homogeneous ideal in a polynomial ring are different has intrigued mathematicians for a number of years.  With many nice applications in Algebraic Geometry and Commutative Algebra, it is not surprising that a number of conjectures about containments between these powers have played a central role in open problems.  
A small sample of work studying symbolic powers includes \cite{refB. et al.}, \cite{refBC}, \cite{refBCH}, \cite{refBH1}, \cite{refBH2}, \cite{refCh}, \cite{refCHT}, \cite{refELS}, \cite{refEM}, \cite{refGHM}, \cite{refHHu}, \cite{refHaHu}, and \cite{refTY}.  

Given a homogeneous ideal $I$ in the polynomial ring $R = k[\mathbb P^n] = k[x_0, \ldots, x_n]$ over the field $k$, the {\it $m$th symbolic power} is defined to be $I^{(m)} = R \cap (\cap_P(I^m)_P)$ where the second intersection is taken over all associated primes $P$ of $I$ and the intersections are taken in the field of fractions of $R$.  It is well-known that $I^{(nm)} \subseteq I^m$ for any $m \geq 1$ (see \cite{refELS}, \cite{refHHu}).  In hopes of tighter containments, Harbourne and Huneke consider the following general question:

\begin{ques}\cite[Question 1.3]{refHaHu}\label{q}
Let $I \subset R = k[x_0, \ldots, x_n]$ be a homogeneous ideal.  For which positive integers $m, i$ and $j$ do we have $I^{(m)} \subseteq M^jI^i$, where $M = (x_0, \ldots, x_n)$ is the irrelevant ideal of $R$?
\end{ques} 

When $I$ is an ideal of a finite set of distinct points, the symbolic power $I^{(m)}$ defines a \emph{fat point scheme}.  In the spirit of Question~\ref{q}, Harbourne and Huneke conjecture several relationships for fat points, including

\begin{conj}\cite[Conjectures 2.1 and 4.1.5]{refHaHu}\label{motivconj}
Let $I \subset R = k[x_0, \ldots, x_n]$ be the ideal of a finite set of distinct points in $\mathbb P^n$ and $M = (x_0, \ldots, x_n) \subset R$ be the irrelevant ideal.  Then 
\begin{enumerate}
\item[(i)] $I^{(rn)} \subseteq M^{r(n-1)}I^r$ holds for all $r > 0$; 
\item[(ii)] $I^{(rn-(n-1))} \subseteq M^{(r-1)(n-1)}I^r$ holds for all $r > 0$.
\end{enumerate}
\end{conj}

Currently there are limited means to verify these conjectures for ideals $I$ of general points in $\mathbb{P}^n$ for $n > 2$ (see \cite{refBCH} for results in $\mathbb P^2$ and \cite{refDu} for results in $\mathbb P^3$).  Thus, looking at special cases has become important in developing tools that can be used in more general situations and in the search for counter-examples.  For example, fat points supported on \emph{star configurations} (which are special point sets given by the pair-wise intersection of hyperplanes) are considered in \cite{refBCH}.  Star configurations have often exhibited extremal behavior and so make a good test case for counter-examples.  In the same spirit we consider fat points supported on \emph{line count configurations} in this paper.  Arising in a number of situations, line count configurations are similar to star configurations and provide us with a large family of points for which we have special tools to work with.  

It is easy to see that if Conjecture~\ref{motivconj} is true, then we have an immediate relationship between the initial degrees in which there are non-zero forms in the symbolic power $I^{(rn)}$ and the regular powers $I^r$ and $M^{r(n-1)}$ for part (i), and in the symbolic power $I^{(rn-(n-1))}$ and the regular powers $I^r$ and $M^{(r-1)(n-1)}$ in part (ii).  In this paper we investigate these relationships on initial degrees so as to provide evidence of Conjecture~\ref{motivconj}.  The conjectures are indeed challenging to prove so it is not surprising that a mixture of tools from Algebraic Geometry and Discrete Mathematics are used.  

The paper is outlined as follows.  Clear statements of the conjectures we pursue are given in Section~\ref{alphaintro}.  Section~\ref{alphaandhilb} then provides the necessary background for symbolic powers of ideals of fat points and their Hilbert functions.  The Hilbert function gives rise to natural bounds on the initial degrees and hence provides a first-step in our investigation.  In Section~\ref{naturalfamily} we further explain why we focus our attention on the special line count configurations.  In Section~\ref{keycase}, we prove a key case that we use in Section~\ref{generalizations} to prove the desired conjectures for any line count configuration of type $(c_1, c_2, \ldots, c_t)$ where $c_i \geq i$.  Finally, in Section~\ref{future}, we discuss future directions.

\subsection{The Alpha Conjectures}\label{alphaintro}

We fix $R:=k[\mathbb P^n] = k[x_0, \ldots, x_n]$ to be the homogeneous coordinate ring for $\mathbb P^n$ over an algebraically closed field $k$ of characteristic 0.

\begin{defn}
Let $P_1, \ldots, P_q$ be distinct points in $\mathbb P^n$ and $m_1, \ldots, m_q$ be non-negative integers.  We denote the ideal generated by all the forms vanishing at the point $P_i$ by $I_{P_i}$.  The ideal $I:=I_{P_1}^{m_1} \cap I_{P_2}^{m_2} \cap \cdots \cap I_{P_q}^{m_q} \subset R$ defines a subscheme $\mathbb Y \subset \mathbb P^n$ which is called a {\it fat point scheme}.
\end{defn}

To ease notation, we write $\mathbb Y := m_1P_1 + \cdots + m_qP_q$ and denote $I$ by $I_{\mathbb Y}$.  If $\mathbb X = \{P_1, \ldots, P_q\}$ and $m_1 = \cdots = m_q = m$ then we simply write $\mathbb Y = m \mathbb X$.  The \emph{support} of $\mathbb Y$ is the set of all $P_i$ for which $m_i > 0$.  We observe that the $m$th symbolic power of $I_{\mathbb Y}$ is $I_{\mathbb Y}^{(m)} = I_{P_1}^{mm_1} \cap I_{P_2}^{mm_2} \cap \cdots \cap I_{P_q}^{mm_q}$.  Thus, when $m_1 = \cdots = m_q = 1$ we see that the $m$th symbolic power of $I_{\mathbb Y}$ gives us the fat point scheme $m \mathbb X$.

The conjectures of interest in this paper give a relationship between the initial degrees of forms in the defining ideals of a fat point scheme and its support.

\begin{defn}
Let $\mathbb Y$ be a fat point scheme in $\mathbb P^n$.  We define the \emph{initial degree of $\mathbb Y$} to be
$$\alpha(\mathbb Y) := \min\{t \in \mathbb Z_{\geq 0} \mid \text{there exists a non-zero form in $I_{\mathbb Y}$ of degree $t$}\}.$$
\end{defn}

We are interested in verifying the following conjectures of Harbourne and Huneke.  These are implications of parts (i) and (ii), respectively, in Conjecture~\ref{motivconj}.

\begin{conj}\label{THEEVENCONJ}
Let $\mathbb W$ be a finite set of distinct points in $\mathbb P^n$ and consider the family of fat point schemes $(rn)\mathbb W \subset \mathbb P^n$ for $r \geq 1$.  For each fat point scheme member of this family, we have the inequality
$$\alpha((rn)\mathbb W) \geq r \alpha(\mathbb W) + r(n-1).$$
\end{conj}

\begin{conj}\label{THEODDCONJ}
Let $\mathbb W$ be a finite set of distinct points in $\mathbb P^n$ and consider the family of fat point schemes $(rn-(n-1))\mathbb W \subset \mathbb P^n$ for $r \geq 1$.  For each fat point scheme member of this family, we have the inequality
$$\alpha((rn-(n-1))\mathbb W) \geq r \alpha(\mathbb W) + (r-1)(n-1).$$
\end{conj}

\subsection{Bounds for Alpha via the Hilbert Function}\label{alphaandhilb}

In this section we briefly recall some of the necessary background of fat points.  From this point onwards we concentrate our investigation in $\mathbb P^2$ and let $R = k[x_0, x_1, x_2] = k[\mathbb P^2]$ where $k$ is an algebraically closed field of characteristic 0.

For any homogeneous ideal $J \subset R$, the quotient $A:=R/J$ is graded:  letting $J_t$ (respectively, $R_t$) be the span of the forms in $J$ (respectively, in $R$) of degree $t$, we see that $A = \oplus_{t \ge 0}A_t$ where $A_t := R_t/J_t$.  The {\it Hilbert function} $H(A)$ gives the $k$-vector space dimension of $A_t$ as a function of $t$.  That is, we define $H(A, t) := \dim_k R_t - \dim_k J_t$ for $t \ge 0$.  We collect these dimensions together in a sequence $H(A) := (H(A, 0), H(A, 1), \ldots)$.  Related to this sequence of dimensions is the {\it first difference Hilbert function} $\Delta H(A)$:  if $H(A) = (h_0, h_1, h_2, \ldots)$, then $\Delta H(A)$ is the sequence $(h_0, h_1 - h_0, h_2 - h_1, \ldots)$.  Note that if $\mathbb Y$ is a fat point scheme, then the ideal $I_{\mathbb Y}$ is homogeneous and so we define the Hilbert function of $\mathbb Y$ to be $H(\mathbb Y) := H(R/I_{\mathbb Y})$ and the first difference Hilbert function of $\mathbb Y$ to be $\Delta H(\mathbb Y)$.  

It will be important to connect $H(\mathbb Y)$ and $\alpha(\mathbb Y)$.  Since $\alpha(\mathbb Y)$ is the initial degree $t$ in which there is a non-zero form in $I_{\mathbb Y}$, we see that 

\begin{align*}
\alpha(\mathbb Y) & = \min\left\{t \in \mathbb Z_{\geq 0} \mid H(\mathbb Y,t) < \dim_k R_t = \binom{t+2}{2} \right\} \\
                  & = \min \left\{t \in \mathbb Z_{\geq 0} \mid \Delta H(\mathbb Y, t) < \dim_k (k[x_0, x_1])_t = t+1 \right\}.
\end{align*}

With this in mind, given any non-negative integer sequence $w = (w_0, w_1, \ldots)$ lying below $\Delta H(k[\mathbb P^2]) = (1, 2, 3, 4, \ldots)$, we will let $\alpha(w)$ be the initial degree for which $w$ is strictly less than the corresponding value of $\Delta H(k[\mathbb P^2])$.  That is, $\alpha(w)$ is the initial degree $t \geq 0$ in which $w_t < t+1$.

For any fat point scheme $\mathbb Y$ in $\mathbb P^2$, results of Cooper--Harbourne--Teitler \cite{refCHT} give us bounding functions for $H(\mathbb Y)$ which naturally bound $\alpha(\mathbb Y)$.  In \cite{refCHT}, the authors study the Hilbert function of a fat point scheme $\mathbb Y \subset \mathbb P^2$ given the multiplicities and information about which subsets of the support are collinear.  A {\it reduction vector $d$} is defined based on the collinearity information and this vector is used to give upper and lower bounds on $H(\mathbb Y)$ in each degree.  We now go into more details of this approach.

\begin{defn}\cite{refCHT}
Let $\mathbb Y = m_1P_1 + m_2P_2 + \cdots + m_qP_q \subset \mathbb P^2$.  Let $\mathbb L$ be a line defined by the linear form $F$. The subscheme 
$$\mathbb Y':= b_1P_1 + b_2P_2 + \cdots + b_qP_q,$$ 
where $b_i = m_i$ if $P_i \not \in \mathbb L$ and $b_i = \max(m_i-1,0)$ if $P_i \in \mathbb L$, is the {\it subscheme of $\mathbb Y$ residual to $\mathbb L$}.  We write $\mathbb Y':=\mathbb Y : \mathbb L$.
\end{defn}

The reduction vector $d$ is then obtained by repeatedly considering residual subschemes.    

\begin{defn}\cite{refCHT} Let $\mathbb Y = m_1P_1 + m_2P_2 + \cdots + m_qP_q \subset \mathbb P^2$.  Let $\mathbb L_1, \mathbb L_2, \ldots, \mathbb L_n$ be a sequence of lines in $\mathbb P^2$, not necessarily distinct.
\begin{enumerate}
\item[(a)] We define fat point schemes $\mathbb Y_0, \mathbb Y_1, \ldots, \mathbb Y_n$ by $\mathbb Y_n := \mathbb Y$ and $\mathbb Y_{j-1} := \mathbb Y_j:\mathbb L_j$.  We say that $\mathbb L_1, \ldots, \mathbb L_n$ {\it totally reduces} $\mathbb Y$ if $\mathbb Y_0$ is empty.
\item[(b)] We also define the vector $d = (d_1, d_2, \ldots, d_n)$ such that $d_k := \sum \{ a_{k,i}: P_i \in \mathbb L_k\}$ where
$$\mathbb Y_k = a_{k,1}P_1 + a_{k,2}P_2 + \cdots + a_{k,q}P_q.$$
The vector $d$ is the {\it reduction vector} for $\mathbb Y$ induced by $\mathbb L_1, \ldots, \mathbb L_n$.
\end{enumerate}
\end{defn}

To illustrate these definitions we give an example.  Note that the order in our residuation process is reversed (but equivalent) from that given in \cite{refCHT}.  We state all results from \cite{refCHT} using the notation given above.

\begin{example}\label{example}
Let $\mathbb Y = \mathbb Y_4 \subset \mathbb P^2$ be the fat point scheme described in the following figure, where the numbers associated to the points represent the multiplicities.  

\hspace{0.4in}\vbox{
\hbox{
\setlength{\unitlength}{0.65cm}
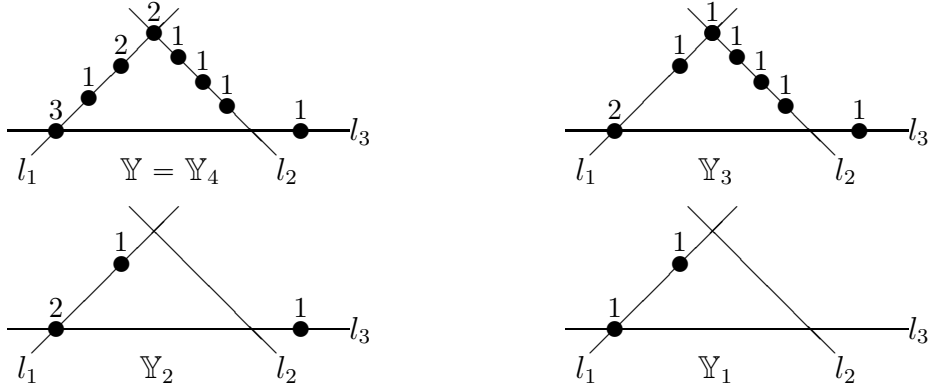
\begin{picture}(4,4)(-1,0)
\put(0,0){\line(1,1){3}}
\put(-.3,-.5){$l_1$}
\put(-.5,.5){\line(1,0){7}}
\put(6.5,.3){$l_3$}
\put(2,3){\line(1,-1){3}}
\put(5,-.5){$l_2$}
\multiput(.5,.5)(.666,.666){3}{\circle*{0.3}}
\multiput(2.5,2.5)(.5,-.5){4}{\circle*{0.3}}
\put(5.5,.5){\circle*{0.3}}
\put(.35,.75){\small 3}
\put(1.016,1.416){\small 1}
\put(1.682,2.082){\small 2}
\put(2.348,2.748){\small 2}
\put(2.848,2.248){\small 1}
\put(3.348,1.748){\small 1}
\put(3.848,1.248){\small 1}
\put(5.35,.75){\small 1}
\put(1.8,-0.5){{$\mathbb Y=\mathbb Y_4$}}
\end{picture}

\setlength{\unitlength}{0.65cm}
\begin{picture}(4,4)(-8,0)
\put(0,0){\line(1,1){3}}
\put(-.3,-.5){$l_1$}
\put(-.5,.5){\line(1,0){7}}
\put(6.5,.3){$l_3$}
\put(2,3){\line(1,-1){3}}
\put(5,-.5){$l_2$}
\put(.5,.5){\circle*{0.3}}
\put(1.833,1.833){\circle*{0.3}}
\put(2.5,2.5){\circle*{0.3}}
\multiput(2.5,2.5)(.5,-.5){4}{\circle*{0.3}}
\put(5.5,.5){\circle*{0.3}}
\put(.35,.75){\small 2}
\put(1.682,2.082){\small 1}
\put(2.348,2.748){\small 1}
\put(2.848,2.248){\small 1}
\put(3.348,1.748){\small 1}
\put(3.848,1.248){\small 1}
\put(5.35,.75){\small 1}
\put(2.2,-0.5){{$\mathbb Y_3$}}
\end{picture}
}

\hbox{
\setlength{\unitlength}{0.65cm}
\begin{picture}(4,4)(-1,0)
\put(0,0){\line(1,1){3}}
\put(-.3,-.5){$l_1$}
\put(-.5,.5){\line(1,0){7}}
\put(6.5,.3){$l_3$}
\put(2,3){\line(1,-1){3}}
\put(5,-.5){$l_2$}
\put(.5,.5){\circle*{0.3}}
\put(1.833,1.833){\circle*{0.3}}
\put(5.5,.5){\circle*{0.3}}
\put(.35,.75){\small 2}
\put(1.682,2.082){\small 1}
\put(5.35,.75){\small 1}
\put(2.2,-0.5){{$\mathbb Y_2$}}
\end{picture}

\setlength{\unitlength}{0.65cm}
\begin{picture}(4,4)(-8,0)
\put(0,0){\line(1,1){3}}
\put(-.3,-.5){$l_1$}
\put(-.5,.5){\line(1,0){7}}
\put(6.5,.3){$l_3$}
\put(2,3){\line(1,-1){3}}
\put(5,-.5){$l_2$}
\put(.5,.5){\circle*{0.3}}
\put(1.833,1.833){\circle*{0.3}}
\put(.35,.75){\small 1}
\put(1.682,2.082){\small 1}
\put(2.2,-0.5){{$\mathbb Y_1$}}
\end{picture}
}
}

\vspace{0.3in}
Given the sequence $\mathbb L_1 = l_1, \mathbb L_2 = l_3, \mathbb L_3 = l_2, \mathbb L_4 = l_1$ of lines $\mathbb L_i$,
$\mathbb Y_{i-1}$ is obtained from $\mathbb Y_i$ by reducing by 1
the multiplicities of all the points on the line $\mathbb L_i$ (and dropping points whose multiplicity becomes $0$).  The output of the reduction procedure with respect to the sequence $\mathbb L_1,\mathbb L_2,\mathbb L_3,\mathbb L_4$
is the reduction vector $d=(2, 3, 4, 8)$: indeed, counting with multiplicity,
$\mathbb Y_4$ has $8$ points on $\mathbb L_4$, $\mathbb Y_3$ has $4$ points on $\mathbb L_3$,
and so on until eventually all of the points have multiplicity 0.  (In \cite{refCHT}, $d$ would have been recorded as $(8, 4, 3, 2)$.)
\end{example}

\begin{thm}\cite{refCHT}
Let $\mathbb Y = m_1P_1+m_2P_2+\cdots+m_qP_q \subset \mathbb P^2$.  If $\mathbb L_1, \ldots, \mathbb L_n$ is a totally reducing sequence of lines for $\mathbb Y$ with associated reduction vector $d$, then there exist sequences $f_d$ and $F_d$ of non-negative integers such that
$$f_d(t) \le H(\mathbb Y, t) \le F_d(t)$$
for each degree $t \geq 0$.
\end{thm}

It turns out that the sequences $f_d$ and $F_d$ are straightforward to compute.  In this paper we only need to know how to compute $f_d$ which can be described via the diagonal count vector.  

\begin{defn}\label{def:standardconfig}\cite{refCHT}
Let $v=(v_1,\dots,v_n)$ be a non-negative integer vector.
We define the {\it standard configuration} $C_v$ determined by $v$
to be the set $C_v\subset {\mathbb Z}\times{\mathbb Z}$ 
of all integer lattice points $(i,j)$ with $i\ge0$ and $0\leq j<n$ such that $i<v_{n-j}$.
Thus $C_v$ consists of the $v_{n-j}$ leftmost first quadrant lattice points
of ${\mathbb Z}\times{\mathbb Z}$ on each 
horizontal line with second coordinate $j$ for $0\leq j<n$.
We also define the {\it diagonal count operator} $diag(v)$ of $v$
by $diag(v) := (p_0,p_1,\dots, p_m)$, where
$p_t$ is the number of points in $C_v$ lying on the diagonal line
with the equation $x+y=t$, and $m$ is the maximum value of $t\geq0$
such that $j+v_{n-j}>t$; i.e., such that the line $x+y=t$ intersects $C_v$.
Equivalently,
\[ p_t = \#\{\, j : 0\leq j < n, i+j=t\text{ and }i < v_{n-j} \, \}. \]
\end{defn}

\begin{example}
Let $v=(1,3,4,5)$.  
Then $C_v$ is as follows
\[
  \begin{array}{ccccc}
    \circ &          &          &          & \\
    \circ & \circ & \circ &           & \\
    \circ & \circ & \circ & \circ & \\
    \circ & \circ & \circ & \circ & \circ
  \end{array}
\]
where $v$ can be regarded as giving the row counts for $C_v$
and $diag(v) = (1, 2, 3, 4, 3)$ gives the diagonal counts (along diagonals with slope $-1$).  By abuse of notation, we will consider $diag(v)$ to be a sequence with infinitely many terms by appending zeros on the end.  Thus, $diag(1,3,4,5) = (1, 2, 3, 4, 3, 0,0,0,0, \ldots)$. 
\end{example}

\begin{fact}\cite{refCHT}
Let $\mathbb Y = m_1P_1 + m_2P_2 + \cdots + m_qP_q \subset \mathbb P^2$.  Suppose $\mathbb L_1, \ldots, \mathbb L_n$ is a totally reducing sequence of lines for $\mathbb Y$ with associated reduction vector $d$.  Let the resulting lower bound on $H(\mathbb Y)$ be $f_d = (f_0, f_1, f_2, \ldots)$ and define $\Delta f_d:=(f_0, f_1 - f_0, f_2-f_1, \ldots)$.  With this notation, we have
$$\Delta f_d = diag(d).$$
\end{fact}

For example, the fat point scheme $\mathbb Y$ in Example \ref{example} has reduction vector $d = (2,3,4,8)$ and $diag(d) = (1,2,3,4,4,1,1,1,0,0,0,\ldots)$.  Thus, 
$$H(\mathbb Y) \ge (1, 3, 6, 10, 14, 15, 16, 17, 17, 17, \ldots) = f_d.$$
As we will see below, it turns out that equality holds here.

In addition to $f_d$ and $F_d$ being easy to compute, there is an easy-to-check combinatorial condition on $d$ which guarantees that $f_d$ is equal to $F_d$, giving us an exact calculation of the Hilbert function in this case.   

\begin{defn}\label{be}\cite{refCHT}
A vector $v := (v_1, \ldots, v_q)$ of non-negative integers is a {\it generalized monotone sequence}, or {\it GMS}, if $v_1 \leq v_2 \leq \cdots \leq v_q$ and between any two zero entries of $\Delta v := (v_1, v_2-v_1, \ldots, v_q-v_{q-1})$ there is an entry which is strictly greater than 1.
\end{defn}

\begin{thm}\cite{refCHT}\label{equality}
Let $\mathbb Y = m_1P_1 + m_2P_2 + \cdots + m_qP_q \subset \mathbb P^2$.  If $\mathbb L_1, \ldots, \mathbb L_n$ is a totally reducing sequence of lines for $\mathbb Y$ with associated reduction vector $d$ such that $d$ is GMS, then
$$f_d = H(\mathbb Y) = F_d.$$
\end{thm}

Even when the reduction vector is not GMS, the bounds $f_d \leq H(\mathbb Y) \leq F_d$ immediately give bounds on $\alpha(\mathbb Y)$.  In particular, we have:

\begin{facts}
Let $\mathbb Y = m_1P_1 + m_2P_2 + \cdots + m_qP_q \subset \mathbb P^2$ and $\mathbb L_1, \ldots, \mathbb L_n$ be a totally reducing sequence of lines for $\mathbb Y$ with associated reduction vector $d$.
\begin{enumerate}
\item We have the lower bound 
$$\alpha(diag(d)) = \alpha(\Delta f_d) \leq \alpha(\mathbb Y) = \alpha(\Delta H(\mathbb Y)).$$
\item Suppose $d$ is the non-negative integer vector $d = (d_1, \ldots, d_s)$.  Then  
$$\alpha(diag(d)) = s + \min(0, d_1-1, d_2-2, \ldots, d_s-s).$$
\end{enumerate}
\end{facts}

\subsection{A Natural Family of Configurations}\label{naturalfamily}  

We will soon focus our study on special sets of points called line count configurations.  It is natural to investigate Conjectures \ref{THEEVENCONJ} and \ref{THEODDCONJ}  via line count configurations since these give us a large family of point sets which arise in a variety of different situations.  

\begin{defn}
A reduced set of points $\mathbb T = \mathbb T_1 \cup \cdots \cup \mathbb T_q \subset \mathbb P^2$ is a {\it line count configuration of type $c = (c_1, \ldots, c_q)$} if each $\mathbb T_i$ consists of $c_i$ points on a line $\mathbb L_i$ where the lines $\mathbb L_1, \ldots, \mathbb L_q$ are distinct and no point of $\mathbb T$ occurs where two of the lines $\mathbb L_i$ meet.  After re-indexing, we assume that $1 \le c_1 \le c_2 \le \cdots \le c_q$.  Moreover, if $a = (a_1, \ldots, a_q)$ is a vector of positive integers, then we define $\mathbb T(a,c) := a_1 \mathbb T_1 + \cdots + a_q \mathbb T_q$ to be the fat point scheme obtained by adding multiplicity $a_i$ to each point in $\mathbb T_i$. 
\end{defn}

Another benefit of working with line count configurations is that the results of \cite{refCHT} nicely bound the associated Hilbert functions.  We will need some notation in order to give the description of these bounds.

\begin{defn}
Given an integer vector $v=(v_1,\ldots,v_n)$,
we define the {\it permuting operator} $\pi(v)$ to be the vector
whose entries are the entries of $v$ permuted to be in non-decreasing order.
\end{defn}

\begin{defn}\cite{refCHT}
Given vectors $a=(a_1,\ldots,a_n)$ and $m=(m_1,\ldots,m_n)$ of positive integers, let
\[ a\circ m := (1m_1,2m_1,\ldots,a_1m_1,\hspace{1ex} 1m_2,2m_2,\ldots,a_2m_2,\hspace{1ex} \ldots, 
\hspace{1ex} 1m_n,2m_n,\ldots,a_n m_n) . \]
The {\it star operator} is defined by $a*m := \pi(a\circ m)$.
\end{defn}

\begin{example}
We have
\[(3,3,3)*(2,4,5) = \pi((2,4,6,\hspace{1ex} 4,8,12,\hspace{1ex} 5,10,15))=(2,4,4,5,6,8,10,12,15).\] 
\end{example}

\begin{lemma}\cite{refCHT}\label{StarLemma}
Let $\mathbb T = \mathbb T_1 \cup \cdots \cup \mathbb T_q \subset \mathbb P^2$ be a line count configuration of type $m = (m_1, \ldots, m_q)$ supported on the lines $\mathbb L_1, \ldots, \mathbb L_q$.  If $a = (a_1, \ldots, a_q)$ is a vector of positive integers, then $a*m$ is the reduction vector for $\mathbb T(a,m)$ with respect to some totally reducing sequence of lines and hence $f_{a*m}\leq H(\mathbb T(a,m)) \leq F_{a*m}$.  Moreover, if $a*m$ is GMS, then $H(\mathbb T(a,m))=f_{a*m}=F_{a*m}$.
\end{lemma}

In the next section, we will work with line count configurations of type $m = (1, 2, 3, \ldots, t)$ in particular.  These line count configurations are special for a number of reasons:  they will play the role of a base case for the main theorem of the paper and, by the following result independently presented by Chudnovsky \cite{refCh} and Geramita--Maroscia--Roberts \cite{refGMR}, they have very special Hilbert functions.

\begin{thm}\cite[Theorem 2.5]{refGMR}
If $\mathbb B$ is any finite set of points in $\mathbb P^2$, then there exists a subset $\mathbb A \subseteq \mathbb B$ such that 
$$\alpha(\mathbb B) = \alpha(\mathbb A) = reg(I_{\mathbb A}).$$
\end{thm}

Let $\alpha = \alpha(\mathbb A)$.  Then $\mathbb A$ must consist of $\binom{\alpha + 1}{2}$ points which have the first difference Hilbert function $diag(1, 2, 3, \ldots, \alpha)$.  That is, $\mathbb A$ has the generic Hilbert function.  To obtain the set $\mathbb A$ we simply repeatedly select points of $\mathbb B$ such that the newest point selected imposes an independent condition on forms of degree $\alpha(\mathbb B)$.  

So, assuming $\alpha(2r \mathbb A) \geq r \alpha(\mathbb A) + r$, then
$$\alpha(2r \mathbb B) \geq \alpha(2r \mathbb A) \geq r \alpha(\mathbb A) + r = r \alpha(\mathbb B) + r.$$

Similarly, assuming $\alpha((2r-1)\mathbb A) \geq r \alpha(\mathbb A) + r-1$, then
$$\alpha((2r-1)\mathbb B) \geq \alpha((2r-1)\mathbb A) \geq r \alpha(\mathbb A) + r -1 = r \alpha(\mathbb B) + r - 1.$$

That is, if Conjectures \ref{THEEVENCONJ} and \ref{THEODDCONJ} hold for $\mathbb A$, then they also hold for $\mathbb B$.  Further observe that line count configurations of type $c = (1, 2, 3, \ldots, t)$ have Hilbert function $diag(1, 2, 3, \ldots, t)$ and that such points have $\alpha = t$.  That is, the Hilbert function of the guaranteed subset $\mathbb A$ has the same form as the Hilbert function of the line count configurations we next consider.

\section{The Key Case}\label{keycase}

In this section we assume $\mathbb W \subset \mathbb P^2$ is a line count configuration of type $c = (1, 2, 3, \ldots, t)$ and consider fat points $\ell \mathbb W$ (where later we will consider the two cases $\ell = 2r-1$ or $\ell=2r$ for some $r \geq 1$).  Our goals are to verify the following inequalities (Conjectures \ref{THEEVENCONJ} and \ref{THEODDCONJ} for $n=2$ and $\mathbb W$):

$$\alpha(2r \mathbb W) \geq r \alpha(\mathbb W) + r$$
and
$$\alpha((2r-1)\mathbb W) \geq r \alpha(\mathbb W) + (r-1).$$

It is easy to see that $H(\mathbb W) = f_c$, and so $\alpha(\mathbb W) = \alpha(diag(c)) = t$.  Furthermore, by Lemma~\ref{StarLemma}, we can find a totally reducing sequence of lines for $\ell \mathbb W$ with associated reduction vector 
$$d = (\ell, \ell, \ldots, \ell) * (1, 2, \ldots, t) := (d_1, d_2, \ldots, d_{\ell t}).$$
Thus $H(\ell \mathbb W) \geq f_d$ and 
$$\alpha(\ell \mathbb W) = \alpha(\Delta H(\ell \mathbb W)) \geq \alpha(f_d) = \alpha(diag(d)) = \ell t +\min(0, d_1-1, d_2-2, \ldots, d_{\ell t}-\ell t).$$
Note that we will prove Conjecture \ref{THEODDCONJ} for $\ell=2r-1$ if we can show that $\alpha(diag(d)) \geq rt+r-1$.  That is, for $\ell = 2r-1$, we need to verify that $\min(0, d_1-1, d_2-2, \ldots, d_{\ell t}-\ell t) \geq r+t-rt-1$.  In the case of $\ell = 2r$, we will prove Conjecture \ref{THEEVENCONJ} if we can show that $\min(0, d_1-1, d_2-2, \ldots, d_{\ell t}-\ell t) \geq r-rt$.  After re-arranging the inequality, we will consider the equivalent bound involving a maximum.  For ease of notation we make the following definition.

\begin{defn}
Throughout this section we will let $\ell$ and $t$ be positive integers and set 
$$d := (\ell, \ell, \ldots, \ell) * (1, 2, \ldots, t) = (d_1, d_2, \ldots, d_{\ell t}).$$
We define the function
$$S(d):= \max(0, 1-d_1, 2-d_2, \ldots, \ell t-d_{\ell t}). $$
\end{defn}

Observe that if $d_{m-1} = d_m$ then \mbox{$(m-1)-d_{m-1} < m-d_m$}.  So, the maximum $S(d)$ occurs at a value $m$ where $d_m < d_{m+1}$.  Also note that if $j$ is an entry of $d$ then the maximum $m$ such that $d_m = j$ is equal to the number of entries in $d$ which are less than or equal to $j$.  Let 
$$\sigma(j):= \#\{\mbox{entries in $d$ which are $\leq j$}\} = \sum_{\substack{1 \leq a \leq \ell \\
1 \leq b \leq t \\
ab \leq j}} 1.$$
We form an upper bound for $S(d)$ by taking the maximum of $\sigma(j)-j$ over all $j$ such that $1 \leq j \leq \ell t$, rather than considering only the entries appearing in $d$.  This clearly gives the bound 
$$S(d) \leq \max_{1 \leq j \leq \ell t}(\sigma(j)-j).$$
In fact, equality holds.  

\begin{lemma}\label{SattainsMax}
$\displaystyle S(d) = \max_{1 \leq j \leq \ell t}(\sigma(j)-j)$.
\end{lemma}

\begin{proof}
Suppose that $\displaystyle \max_{1 \leq j \leq \ell t}(\sigma(j)-j)$ is attained at some $\hatj \not \in d$.   Let $d_{\hati}$ be the largest $d_i < \hatj$.  Then $\sigma(\hatj) = \sigma(d_{\hati})$ and so
$$\sigma(\hatj)-\hatj < \sigma(d_{\hati}) - d_{\hati},$$
contradicting that $\displaystyle \max_{1 \leq j \leq \ell t}(\sigma(j)-j)$ is attained at $\hatj$.
\end{proof}

For $1 \leq b \leq t$, let $D_b := \{ab : 1 \leq a \leq \ell\}$.  Note that $d = \cup_{b=1}^t D_b$ as an unordered multiset.  Let $\chi_{b}(j):=\left|D_{b}\cap\{1,2,\ldots,j\}\right|$;
then 
\begin{align*}
\sigma(j) & =\sum_{b=1}^{t}\left|D_{b}\cap\{1,2,\ldots,j\}\right|=\sum_{b=1}^{t}\chi_{b}(j).
\end{align*}
Note that 
\[
\chi_{b}(j)=\begin{cases}
\lfloor j/b \rfloor & \textrm{if $j<b\ell$,}\\
\ell & \textrm{if $j\ge b\ell$.}\end{cases}\]
Thus, we are interested in 
$$S(d) = \max_{1 \leq j \leq \ell t}\left(\left(\sum_{b=1}^t \chi_{b}(j)\right) - j\right).$$
We focus our attention on the function 
$$\Phi_t(j) := \left(\sum_{b=1}^t \chi_b(j)\right) - j.$$
Our main theorem hinges on the following bounds for $\Phi_t(j)$:

\begin{claim}\label{THECLAIM}  
Let $t$, $\ell$, and $j$ be positive integers.  If one of $t$ or $\ell$ is 2 and the other is even, then
$$\Phi_t(j) \leq \frac{1}{2} \ell(t-1).$$
Otherwise,
$$\Phi_t(j) \leq \frac{1}{2} (\ell-1)(t-1).$$
\end{claim}

The proof of Claim \ref{THECLAIM} will follow from a series of preliminary bounds, namely Lemmas \ref{jbig}---\ref{l=2} and Proposition~\ref{middlej}.  The proof presented here was suggested by the anonymous referee.  Our original proof used Dirichlet's hyperbola method and determined the values of $j$ where $\Phi_t(j)$ is maximized.  A manuscript containing the original proof is posted on the arXiv.\footnote{\url{http://arxiv.org/abs/1312.4147v1}}

\begin{lemma}\label{jbig}
If $j \geq \ell t$, then $\Phi_t(j) \leq \frac{1}{2}(\ell-1)(t-1)$.
\end{lemma}

\begin{proof}
We have $j \geq \ell t \geq \ell(t-1) \geq \cdots \geq \ell$.  So, by the definition of $\chi_b(j)$, 
$$\Phi_t(j) = \ell t - j \leq 0 \leq \frac{1}{2}(\ell - 1 )(t-1).$$ 
\end{proof}

\begin{lemma}\label{jsmall}
If $1 \leq j < \ell$, then $\Phi_t(j) \leq \frac{1}{2}(\ell - 1)(t-1)$.
\end{lemma}

\begin{proof}
For each $1 \leq b \leq t$, we have $\chi_b(j) = \lfloor j/b \rfloor$.  Thus, since $\frac{j}{t} \leq \frac{j}{t-1} \leq \cdots \leq \frac{j}{2}$, we have 
$$\Phi_t(j) = \left(\sum_{b=1}^t \chi_b(j)\right) - j = \left(j + \left\lfloor \frac{j}{2} \right\rfloor + \cdots + \left\lfloor \frac{j}{t} \right\rfloor\right) - j \leq \sum_{b=2}^t\frac{j}{b} \leq \frac{j}{2}(t-1) \leq \frac{1}{2}(\ell - 1)(t-1).$$
\end{proof}

\begin{lemma}\label{t=2}
Let $t = 2$ and suppose $\ell$ and $j$ are positive integers such that $\ell \leq j < 2 \ell$.  We have the following bounds:
\begin{itemize}
\item if $\ell = 2m+1$ for some positive integer $m$, then $\Phi_t(j) \leq \frac{1}{2}(\ell - 1)$;
\item if $\ell = 2m$ for some positive integer $m$, then $\Phi_t(j) \leq \frac{1}{2}\ell$.
\end{itemize}
\end{lemma}

\begin{proof}
Since $\ell \leq j < 2 \ell$, we have
$$\Phi_t(j) = \ell + \left\lfloor \frac{j}{2} \right\rfloor - j \leq \ell + \frac{j}{2} - j = \ell - \frac{j}{2} \leq \ell - \frac{\ell}{2} = \frac{1}{2} \ell.$$
Moreover, if $\ell = 2m+1$ for some integer $m$ then, since $\Phi_t(j)$ is an integer,
$$\Phi_t(j) \leq \frac{1}{2}(\ell-1).$$
\end{proof}

\begin{lemma}\label{l=2}
Let $\ell = 2, t \geq 3$ and $2 \leq j \leq 2t-1$.  
\begin{itemize}
\item If $t \geq 4$ is even, then $\Phi_t(j) \leq t-1$;
\item if $t \geq 3$ is odd, then $\Phi_t(j) \leq \frac{1}{2}(t-1)$.
\end{itemize}
\end{lemma}

\begin{proof}
Let $k \in \{1, \ldots, t-1\}$ be an integer, and let $2k \leq j \leq 2k+1$.  Note that by definition,
$$\Phi_t(j) = \left(\sum_{b=1}^t \left\lfloor \frac{j}{b} \right\rfloor\right) - j = (2k-j) + \left(\sum_{b=k+1}^t \left\lfloor \frac{j}{b} \right\rfloor\right).$$
We first assume that $t \geq 4$ is even.  Observe that $\frac{j}{k+1} \leq \frac{2k+1}{k+1} < 2$, and so $\left\lfloor \frac{j}{k+1} \right\rfloor \leq 1$.  Thus, by definition and since $\left\lfloor \frac{j}{t} \right\rfloor \leq \left\lfloor \frac{j}{t-1} \right\rfloor \leq \cdots \leq \left\lfloor \frac{j}{k+1} \right\rfloor$, we have

$$\Phi_t(j) = (2k-j) + \left(\sum_{b=k+1}^t \left\lfloor \frac{j}{b} \right\rfloor\right) \leq (2k-j) + (t-k) \left\lfloor \frac{j}{k+1} \right\rfloor \leq 0 + (t-k)(1) \leq t-1.
$$

Now assume that $t \geq 3$ is odd.  We consider two cases:
\begin{itemize}
\item Assume $j= 2k$.  We have 
$$\Phi_t(j) = \left(\sum_{b=k+1}^t \left\lfloor \frac{2k}{b} \right\rfloor\right).$$
Suppose first that $k \geq \frac{t-1}{2} + 1 = \frac{t+1}{2}$.  Then $t \leq 2k-1$, and so $\left\lfloor \frac{2k}{b} \right\rfloor = 1$ for $k+1 \leq b \leq t$.  Thus, 
$$\Phi_t(j) = t-k \leq t - \left(\frac{t+1}{2}\right) = \frac{2t-t-1}{2} = \frac{t-1}{2}.$$
Secondly, if $k \leq \frac{t-1}{2}$, then $t \geq 2k+1$, and so 
$$\Phi_t(j) = \sum_{b=k+1}^{2k}\left\lfloor \frac{2k}{b} \right\rfloor + \sum_{b=2k+1}^{t}\left\lfloor \frac{2k}{b} \right\rfloor = \sum_{b=k+1}^{2k} 1 + \sum_{b=2k+1}^t 0 = k \leq \frac{t-1}{2}.$$
\item Assume $j = 2k+1$.  In this case, 
$$\Phi_t(j) = \left(\sum_{b=k+1}^t \left\lfloor \frac{2k+1}{b} \right\rfloor\right) - 1.$$
Suppose first that $k \geq \frac{t-1}{2}$.  Then $t \leq 2k+1$, and so $\left\lfloor \frac{2k+1}{b} \right\rfloor = 1$ for $k+1 \leq b \leq t$.  Thus, 
$$\Phi_t(j) = t-k-1 \leq t - \left(\frac{t-1}{2}\right) -1 = \frac{2t-t+1-2}{2} = \frac{t-1}{2}.$$
Secondly, if $k < \frac{t-1}{2}$, then $t \geq 2k+2$, and so 
\begin{align*}
\Phi_t(j) = \left(\sum_{b=k+1}^{2k+1}\left\lfloor \frac{2k+1}{b} \right\rfloor \right)+ \left(\sum_{b=2k+2}^{t} \left\lfloor \frac{2k+1}{b} \right\rfloor \right) - 1 & = \left(\sum_{b=k+1}^{2k+1} 1 \right) + \left(\sum_{b=2k+2}^t 0\right)  - 1 \\ 
          & = (k+1)-1 \\
          & < \frac{t-1}{2}.
\end{align*}
\end{itemize}

\end{proof}

Finally, we investigate the critical values for $j$.

\begin{prop}\label{middlej}
Let $t \geq 3$ and $\ell \geq 3$.  If $\ell \leq j < \ell t$, then $\Phi_t(j) \leq \frac{1}{2}(\ell - 1)(t-1)$.
\end{prop} 

\begin{proof}
Let $k$ be the integer such that $k \ell \leq j < (k+1) \ell$.  Note that $1 \leq k \leq t-1$.

We proceed by induction on $t$.  To this end, first suppose that $t = 3$.  Then $k \in \{1, 2\}$.  Fix $k = 1$.  In this case, $3 \leq \ell \leq j \leq 2 \ell - 1$.  Thus, 
$$\Phi_3(j) = \ell + \left\lfloor \frac{j}{2} \right\rfloor + \left\lfloor \frac{j}{3} \right\rfloor - j \leq \ell + \frac{j}{2} + \frac{j}{3} - j = \ell - \frac{j}{6}.$$
Observe that $\ell - \frac{j}{6} \leq \ell - 1 = \frac{1}{2}(\ell - 1)(t-1)$ if and only if $j \geq 6$.  That is, the claim is true for $j \geq 6$ and hence for $\ell \geq 6$.  We consider the cases $\ell = 3, 4$, and $5$  with $j \leq 5$ separately.
\begin{itemize}
\item Let $\ell = 3$ and $j \in \{3, 4, 5\}$.
\begin{itemize}
\item If $j = 3$, then $\Phi_3(j) = 3 + \lfloor \frac{3}{2}  \rfloor + \lfloor \frac{3}{3} \rfloor - 3 = 2 \leq \ell-1 = 2$.
\item If $j = 4$, then $\Phi_3(j) = 3 + \lfloor \frac{4}{2}  \rfloor + \lfloor \frac{4}{3} \rfloor - 4 = 2 \leq \ell-1 = 2$.
\item If $j = 5$, then $\Phi_3(j) = 3 + \lfloor \frac{5}{2}  \rfloor + \lfloor \frac{5}{3} \rfloor - 5 = 1 \leq \ell-1 = 2$.
\end{itemize}
\item Let $\ell = 4$ and $j \in \{4, 5\}$.
\begin{itemize}
\item If $j = 4$, then $\Phi_3(j) = 4 + \lfloor \frac{4}{2}  \rfloor + \lfloor \frac{4}{3} \rfloor - 4 = 3 \leq \ell-1 = 3$.
\item If $j = 5$, then $\Phi_3(j) = 4 + \lfloor \frac{5}{2}  \rfloor + \lfloor \frac{5}{3} \rfloor - 5 = 2 \leq \ell-1 = 3$.
\end{itemize}
\item Let $\ell = 5$ and $j = 5$.  In this case, $\Phi_3(j) = 5 + \lfloor \frac{5}{2}  \rfloor + \lfloor \frac{5}{3} \rfloor - 5 = 3 \leq \ell-1 = 4$.
\end{itemize}
In each of these cases, we see that $\Phi_3(j) \leq \frac{1}{2}(\ell - 1)(t-1)$ as desired.

Now fix $k = 2$ so that $6 \leq 2 \ell \leq j \leq 3 \ell - 1$.  Hence,
$$\Phi_3(j) = 2 \ell + \left\lfloor \frac{j}{3} \right\rfloor - j \leq 2 \ell + \frac{j}{3} - j = 2 \ell - \frac{2}{3}j \leq 2 \ell - \frac{4}{3} \ell = \frac{2}{3} \ell \leq \ell - 1$$
since $\ell \geq 3$.  That is, the claim is true for $k = 2$.  This completes the proof of our base case $t = 3$.  

\medskip
We now set $t \geq 4$ and assume that the assertion holds for $t-1$.  We consider three cases depending on the value of $k$.

\medskip
\noindent
\emph{Case 1:}  Assume $1 \leq k \leq \frac{1}{2}(t-2)$.  Since $j < (k+1)\ell$, we have $j < \frac{1}{2} \ell t$ and so $\frac{j}{t} < \frac{1}{2} \ell$ which implies that $\lfloor \frac{j}{t} \rfloor \leq \frac{1}{2}(\ell - 1)$.  Therefore, by induction,
$$\Phi_t(j) = \Phi_{t-1}(j) + \left\lfloor \frac{j}{t} \right\rfloor \leq \frac{1}{2}(\ell - 1)(t-2) + \left\lfloor \frac{j}{t} \right\rfloor \leq \frac{1}{2}(\ell - 1)(t-2) + \frac{1}{2}(\ell - 1) = \frac{1}{2}(\ell - 1)(t-1).$$

\medskip
\noindent
\emph{Case 2:}  Assume $\frac{1}{2}(t+1) \leq k \leq t-1$.  Observe that $\frac{j}{t} \leq \frac{j}{t-1} \leq \cdots \leq \frac{j}{k+1} < \ell$ and so $\lfloor \frac{j}{t} \rfloor \leq \lfloor \frac{j}{t-1} \rfloor \leq \cdots \leq \lfloor \frac{j}{k+1} \rfloor \leq \ell - 1$.  Thus,
$$\Phi_t(j) = \left(k \ell + \sum_{k+1}^t \left\lfloor \frac{j}{b} \right\rfloor \right) - j \leq k \ell + (t-k)(\ell - 1) - j \leq (t-k)(\ell - 1).$$
Now, since $\frac{1}{2}(t+1) \leq k$, we have 
$$\Phi_t(j) \leq (t-k)(\ell-1) \leq \left(t-\frac{t}{2} - \frac{1}{2} \right)(\ell-1) = \frac{1}{2}(\ell-1)(t-1),$$
as desired.

\medskip
\noindent
\emph{Case 3:}  Assume $\frac{1}{2}(t-2) < k < \frac{1}{2}(t+1)$.  If $t$ is even, let $t = 2m$ with $m \geq 2$; then $k = m = \frac{1}{2}t$.  If $t$ is odd, let $t = 2m+1$ with $m \geq 1$; then $k = m = \frac{1}{2}(t-1)$.  We consider these two cases separately.

\begin{itemize}
\item Suppose $k = \frac{1}{2}t$.  In this case, $k = \frac{1}{2}t \geq \frac{4}{2} = 2$.  We consider two situations.
\begin{enumerate}
\item[(3a)] Let $j = k \ell$.  Then we have the following sequence of equivalences
$$\frac{k \ell}{2k-1} \geq \ell - 1 \iff 2k-1 \geq \ell(k-1) \geq 3(k - 1) \iff k \leq 2 \iff k = 2.$$
Thus, if $k \geq 3$, we have
\begin{align*}
\Phi_t(k \ell) & = \left(k \ell + \sum_{b=k+1}^{2k-2} \left\lfloor \frac{k \ell}{b} \right\rfloor + \left\lfloor \frac{k \ell}{2k-1} \right\rfloor + \left\lfloor \frac{k \ell}{2k} \right\rfloor \right) - k \ell \\
               & \leq (k-2)(\ell-1) + (\ell - 2) + \frac{1}{2} \ell \\
               & = (k-2)(\ell-1) + (\ell-1) - 1 + \frac{1}{2} \ell \\
               & = (k-1)(\ell-1) - 1 + \frac{1}{2} \ell \\
               & < (k-1)(\ell-1) + \frac{1}{2} (\ell - 1) \\
               & = \frac{1}{2}(\ell-1)(2k-1) \\
               & = \frac{1}{2}(\ell-1)(t-1),
\end{align*}
where the first inequality uses the obvious inequalities and hypothesis that 
$$\frac{k \ell}{2k-2} \leq \frac{k \ell}{2k-3} \leq \cdots \leq \frac{k \ell}{k+1} = \frac{j}{k+1} < \ell.$$    

The case $k=2$ remains.  Here $t=4$ and 
$$\Phi_t(2 \ell) = \left(2 \ell + \left\lfloor \frac{2 \ell}{3} \right\rfloor + \left\lfloor \frac{2 \ell}{4} \right\rfloor \right) - 2 \ell \leq \frac{2 \ell}{3} + \frac{\ell}{2} = \frac{7 \ell}{6}.$$
We are done if $\ell \geq 5$ since 
$$\frac{7 \ell}{6} \leq \frac{1}{2} (4-1) (\ell-1) = \frac{3}{2} (\ell-1) \iff \ell \geq 5.$$
We handle the cases $\ell = 3$ and $\ell = 4$ directly:
\begin{itemize}
\item If $\ell = 3$, then $\Phi_t(2 \ell) = \left(2(3) + \left\lfloor \frac{2(3)}{3} \right\rfloor + \left\lfloor \frac{2(3)}{4} \right\rfloor \right) - 2(3) = \left\lfloor 2 \right\rfloor + \left\lfloor \frac{3}{2} \right\rfloor = 3 \leq \frac{1}{2}(3-1)(4-1) = \frac{1}{2}(\ell-1)(t-1)$.
\item If $\ell = 4$, then $\Phi_t(2 \ell) = \left(2(4) + \left\lfloor \frac{2(4)}{3} \right\rfloor + \left\lfloor \frac{2(4)}{4} \right\rfloor \right) - 2(4) = \left\lfloor \frac{8}{3} \right\rfloor + \left\lfloor 2 \right\rfloor = 4 \leq \frac{1}{2}(4-1)(4-1) = \frac{1}{2}(\ell-1)(t-1)$.
\end{itemize}

\item[(3b)] Assume $j \geq k \ell + 1$.  Since $t = 2k$, we want to show that 
$$\Phi_t(j) = \left(k \ell + \sum_{b=k+1}^{2k} \left\lfloor \frac{j}{b} \right\rfloor \right) - j \leq \frac{1}{2}(\ell - 1)(2k-1) = (k-1)(\ell-1) + \frac{1}{2}(\ell-1).$$
Now $\frac{j}{2k-1} \leq \frac{j}{2k-2} \leq \cdots \leq \frac{j}{k+1} < \ell$, and so $\sum_{b=k+1}^{2k-1} \left\lfloor \frac{j}{b} \right\rfloor \leq (k-1)(\ell-1)$. Thus, it suffices to show that 
\begin{equation*}
\left\lfloor \frac{j}{2k} \right\rfloor \leq \frac{1}{2}(j-k \ell) + \frac{1}{2}(\ell-1).\label{eq:inequality}\tag{$\ast$}
\end{equation*}
We will do this by considering $\ell$ even and $\ell$ odd.  However, before proceeding, we verify two necessary bounds.

\medskip
\noindent
\emph{Necessary Bound 1:}  We claim 
$$\left\lfloor \frac{j-k \ell}{2k} \right\rfloor \leq \frac{j-k\ell}{2} - \frac{1}{2}.$$
To see this, let $a := j-k \ell \geq 1$ and $b := \frac{a}{2} \geq \frac{1}{2}$.  Then
$$\left\lfloor \frac{j-k \ell}{2k} \right\rfloor = \left\lfloor \frac{b}{k} \right\rfloor \leq \left\lfloor \frac{b}{2} \right\rfloor.$$
If $b \geq 1$, then
$$\left\lfloor \frac{j-k \ell}{2k} \right\rfloor \leq \left\lfloor \frac{b}{2} \right\rfloor \leq \frac{b}{2} \leq \frac{b}{2} + \left(\frac{b}{2} - \frac{1}{2} \right) = b - \frac{1}{2} = \frac{j-k \ell}{2} - \frac{1}{2}.$$
If $\frac{1}{2} \leq b < 1$, then 
$$\left\lfloor \frac{j-k \ell}{2k} \right\rfloor \leq \left\lfloor \frac{b}{2} \right\rfloor = 0 \leq b - \frac{1}{2} = \frac{j-k \ell}{2} - \frac{1}{2}.$$

\medskip
\noindent
\emph{Necessary Bound 2:}  We have
$$\left\lfloor \frac{j-k \ell + k}{2k} \right\rfloor \leq \frac{j-k\ell}{2}.$$
To see this, again let $a := j-k \ell \geq 1$ and $b := \frac{a}{2} \geq \frac{1}{2}$.  Then
$$\left\lfloor \frac{j-k \ell+k}{2k} \right\rfloor = \left\lfloor \frac{a}{2k} + \frac{k}{2k}\right\rfloor = \left\lfloor \frac{b}{k} + \frac{1}{2} \right\rfloor \leq \left\lfloor \frac{b}{2} + \frac{1}{2} \right\rfloor = \left\lfloor \frac{b+1}{2} \right\rfloor.$$
If $b \geq 1$, then
$$\left\lfloor \frac{j-k \ell+k}{2k} \right\rfloor \leq \left\lfloor \frac{b+1}{2} \right\rfloor \leq \frac{b+1}{2} \leq \frac{b+b}{2} = b = \frac{a}{2} = \frac{j-k \ell}{2}.$$
If $\frac{1}{2} \leq b < 1$, then 
$$\left\lfloor \frac{j-k \ell + k}{2k} \right\rfloor \leq \left\lfloor \frac{b+1}{2} \right\rfloor = 0 < b = \frac{j-k \ell}{2}.$$

\medskip
We now proceed with proving inequality~\eqref{eq:inequality}.  First let $\ell$ be even.  Then by Necessary Bound~1,
\begin{align*}
\left\lfloor \frac{j}{2k} \right\rfloor = \left\lfloor \frac{j- k \ell + k \ell}{2k} \right\rfloor = \left\lfloor \frac{\ell}{2} + \frac{j-k \ell}{2k} \right\rfloor & = \frac{\ell}{2} + \left\lfloor \frac{j - k \ell}{2k} \right\rfloor \\
                                        & \leq \frac{\ell}{2} + \frac{j-k \ell}{2} - \frac{1}{2} \\
                                        & = \frac{j-k \ell}{2} + \frac{\ell-1}{2}.
\end{align*}
Similarly, if $\ell$ is odd, then by Necessary Bound~2, 
\begin{align*}
\left\lfloor \frac{j}{2k} \right\rfloor = \left\lfloor \frac{j - k (\ell-1) + k (\ell-1)}{2k} \right\rfloor & = \left\lfloor \frac{\ell-1}{2} + \frac{j-k \ell + k}{2k} \right\rfloor\\
 & = \frac{\ell - 1}{2} + \left\lfloor \frac{j - k \ell + k}{2k} \right\rfloor \\
                                        & \leq \frac{\ell - 1}{2} + \frac{j-k \ell}{2},
\end{align*}
which completes this case.
\end{enumerate}

\item Suppose $k = \frac{1}{2}(t-1)$.  In this case we have $t=2k+1$ and so we need to show that 
\begin{align*}
\Phi_t(j) = \left(k \ell + \sum_{b=k+1}^{2k+1} \left\lfloor \frac{j}{b} \right\rfloor \right) - j & = \left(\sum_{b=k+1}^{2k+1} \left\lfloor \frac{j}{b} \right\rfloor \right) - (j-k \ell) \\
& \leq \frac{1}{2}(\ell-1)(t-1) = \frac{1}{2}(\ell-1)(2k) = k(\ell-1).
\end{align*}
We again consider two situations.

\begin{enumerate}
\item[(3A)] Let $j = k \ell$.  As in Case (3a), we have the equivalence
$$\frac{k \ell}{2k-1} \geq \ell-1 \iff k \leq 2.$$
So, if $k \geq 3$, then $\lfloor \frac{k \ell}{2k-1} \rfloor \leq \ell-2$ which implies that 
\begin{align*}
\Phi_t(k \ell) & = \left(k \ell + \sum_{b=k+1}^{2k-2} \left\lfloor \frac{k \ell}{b} \right\rfloor + \left\lfloor \frac{k \ell}{2k-1} \right\rfloor + \left\lfloor \frac{k \ell}{2k} \right\rfloor + \left\lfloor \frac{k \ell}{2k+1} \right\rfloor \right) - k \ell \\
               & \leq (k-2)(\ell-1) + (\ell - 2) + \frac{1}{2} \ell + \frac{1}{2} \ell\\
               & = (k-2)(\ell-1) + (\ell-1) + (\ell-1) \\
               & = k(\ell-1),
\end{align*}
as desired.  So we may assume that $k \leq 2$.  
\begin{itemize}
\item If $k=2$ and $\ell \geq 5$, then $\Phi_t(2 \ell) = \left\lfloor \frac{2 \ell}{3} \right\rfloor +  \left\lfloor \frac{2 \ell}{4} \right\rfloor +  \left\lfloor \frac{2 \ell}{5} \right\rfloor \leq \frac{2 \ell}{3} + \frac{2 \ell}{4} + \frac{2 \ell}{5} = \frac{47}{30} \ell \leq k(\ell-1)$.
\item If $k = 2$ and $\ell = 3$, then $\Phi_t(2 \ell) = \left\lfloor \frac{6}{3} \right\rfloor +  \left\lfloor \frac{6}{4} \right\rfloor +  \left\lfloor \frac{6}{5} \right\rfloor = 4 \leq k(\ell-1)$.
\item If $k = 2$ and $\ell = 4$, then $\Phi_t(2 \ell) = \left\lfloor \frac{8}{3} \right\rfloor +  \left\lfloor \frac{8}{4} \right\rfloor +  \left\lfloor \frac{8}{5} \right\rfloor = 5 \leq k(\ell-1)$.
\item If $k=1$ and $\ell \geq 6$, then $\Phi_t(\ell) = \left\lfloor \frac{\ell}{2} \right\rfloor +  \left\lfloor \frac{\ell}{3} \right\rfloor \leq \frac{\ell}{2} + \frac{\ell}{3} = \frac{5}{6} \ell \leq k(\ell-1)$.
\item If $k = 1$ and $\ell = 3$, then $\Phi_t(\ell) = \left\lfloor \frac{3}{2} \right\rfloor +  \left\lfloor \frac{3}{3} \right\rfloor = 2 \leq k(\ell-1)$.
\item If $k = 1$ and $\ell = 4$, then $\Phi_t(\ell) = \left\lfloor \frac{4}{2} \right\rfloor +  \left\lfloor \frac{4}{3} \right\rfloor = 3 \leq k(\ell-1)$.
\item If $k = 1$ and $\ell = 5$, then $\Phi_t(\ell) = \left\lfloor \frac{5}{2} \right\rfloor +  \left\lfloor \frac{5}{3} \right\rfloor = 3 \leq k(\ell-1)$.
\end{itemize}
This completes Case (3A).

\item[(3B)] Suppose $j \geq k \ell+1$.  If $k = 1$, then $t = 3$ and we are done by the base case.  Thus, we assume $k \geq 2$.  The argument is essentially the same as for Case (3b).  We want to show that 
$$\Phi_t(j) \leq k(\ell-1) = (k-1)(\ell-1) + \ell -1.$$
We know that $\sum_{b=k+1}^{2k-1} \left\lfloor \frac{j}{b} \right \rfloor \leq (k-1)(\ell-1)$.  Thus, it suffices to show that 
$$\left\lfloor \frac{j}{2k} \right\rfloor + \left\lfloor \frac{j}{2k+1} \right\rfloor \leq \ell - 1 + j - k \ell.$$
From inequality~\eqref{eq:inequality} in Case (3b), we have that
$$\left\lfloor \frac{j}{2k+1} \right\rfloor \leq \left\lfloor \frac{j}{2k} \right\rfloor \leq \frac{1}{2}(j-k \ell) + \frac{1}{2}(\ell-1).$$
Thus,
$$\left\lfloor \frac{j}{2k} \right\rfloor + \left\lfloor \frac{j}{2k+1} \right\rfloor \leq 2 \left(\frac{1}{2}(j-k \ell) + \frac{1}{2}(\ell-1) \right) = j - k \ell + \ell - 1.$$
\end{enumerate}
\end{itemize} 
\end{proof}

Putting all of the above bounds together yields Claim \ref{THECLAIM}.

\begin{thm}\label{themainthm}
Let $t$, $\ell$, and $j$ be positive integers.  If one of $t$ or $\ell$ is 2 and the other is even, then
$$\Phi_t(j) \leq \frac{1}{2} \ell(t-1).$$
Otherwise,
$$\Phi_t(j) \leq \frac{1}{2} (\ell-1)(t-1).$$
\end{thm}

\begin{proof}
This theorem is an accumulation of Lemmas \ref{jbig}---\ref{l=2} and Proposition \ref{middlej}. 
\end{proof}

We can now prove Conjectures \ref{THEEVENCONJ} and \ref{THEODDCONJ} for $n=2$ and line count configurations $\mathbb W$ of type $c = (1, 2, 3, \ldots, t)$.

\begin{thm}
Let $\mathbb W \subset \mathbb P^2$ be a line count configuration of type $c = (1, 2, 3, \ldots, t)$.  For all integers $r \geq 1$ we have 
\begin{enumerate}
\item[(a)] $\alpha((2r-1)\mathbb W) \geq r \alpha(\mathbb W) + (r-1)$;
\item[(b)] $\alpha(2r \mathbb W) \geq r \alpha(\mathbb W) + r$.
\end{enumerate}
\end{thm}

\begin{proof}
We know that 
$$\alpha(\ell \mathbb W) = \alpha(\Delta H(\ell \mathbb W)) \geq \alpha(f_d) = \alpha(diag(d)) = \ell t +\min(0, d_1-1, d_2-2, \ldots, d_{\ell t}-\ell t),$$
where $\ell = 2r-1$ in case (a) and $\ell = 2r$ in case (b).  Thus, for case (a) we need to verify that $\min(0, d_1-1, d_2-2, \ldots, d_{\ell t}-\ell t) \geq r+t-rt-1$.  In the case of (b), we need to verify that $\min(0, d_1-1, d_2-2, \ldots, d_{\ell t}-\ell t) \geq r-rt$.  Equivalently, for case (a) we need to verify 
$$S(d) = \max(0, 1-d_1, 2-d_2, \ldots, \ell t-d_{\ell t}) \leq \frac{1}{2}(\ell-1)(t-1).$$
For case (b), we equivalently need to show that 
$$S(d) = \max(0, 1-d_1, 2-d_2, \ldots, \ell t-d_{\ell t}) \leq \frac{1}{2}\ell(t-1).$$
Both bounds follow by Theorem \ref{themainthm}.
\end{proof}

\section{Generalization of the Key Case}\label{generalizations}

In the previous section we proved Conjectures \ref{THEEVENCONJ} and \ref{THEODDCONJ} for line count configurations $\mathbb W \subset \mathbb P^2$ of type $c = (1, 2, \ldots, t)$.  It is natural to try to generalize this result in a number of different directions.  In this section we consider line count configurations of type $c = (c_1, \ldots, c_t)$ where $c_i \geq i$ for each $i$.

To work on this more general case, we generalize $S(d)$ from the previous section.

\begin{defn}
For an integer vector $v = (v_1, \ldots, v_m)$ whose entries are non-decreasing, we define the function
$$S(v):= \max(0, 1-v_1, 2-v_2, \ldots, m-v_m). $$
\end{defn}

We will need the following preliminary lemmas.  

\begin{lemma}\label{genprelemma}
Let $v = (v_1, \ldots, v_m)$ be an integer vector whose entries are non-decreasing, and let $1\le i \le m$ be an index such that $v_{i-1}<v_i$ if $i>1$.  Define $v' := (v_1,\ldots,v_{i-1},v_i-1,\ldots,v_m)$.  Then $S(v) \leq S(v')$.
\end{lemma}

\begin{proof}
Note that the entries of $v'$ are non-decreasing, and hence $S(v')$ is defined. The arguments over which the maximums $S(v)$ and $S(v')$ are being taken have the same corresponding entries, except at index $i$, where the entry of $v'$ is $1$ greater than the $i$th entry of $v$. Hence, $S(v') \geq S(v)$.
\end{proof}

\begin{lemma}\label{genprelemma2}
Let $v = (v_1, \ldots, v_m)$ be an integer vector whose entries are non-decreasing, and let $w$ be an integer vector (not necessarily non-decreasing) such that $w_i\le v_i$ for $1\le i \le m$.  
Then $S(v) \leq S(\pi(w))$.
\end{lemma}

\begin{proof}
First note that $\pi(w)_i \le v_i$ for $1\le i \le m$.  Otherwise, suppose that $i$ is the least index where $\pi(w)_i > v_i$.  Let $i'$ be the index of $\pi(w)_i$ before $\pi$ was applied.  If $i'>i$ then $\pi(w)_i = w_{i'}\le w_i \le v_i$. If $i'<i$ then $\pi(w)_i = w_{i'} \le v_{i'} \le v_{i}$.  Both cases give a contradiction.

We iteratively transform $v$ to $\pi(w)$ by reducing one entry by $1$ at each stage.  Set $v^{(0)}:=v$.  For $j\ge0$, let $i$ be the least index such that $v^{(j)}_i > \pi(w)_i$.  Define $v^{(j+1)} := (v^{(j)}_1,\ldots,v^{(j)}_{i-1},v^{(j)}_i-1,\ldots,v^{(j)}_m)$.  Note that if $i>1$, then $v^{(j+1)}_{i-1}=\pi(w)_{i-1}\le \pi(w)_i < v^{(j)}_i$.  Hence Lemma~\ref{genprelemma} applies, and $S(v^{(j)})\le S(v^{(j+1)})$.

As an entry is reduced by $1$ at each stage, this process terminates at some step $k$ where $v^{(k)}=\pi(w)$.  Thus, 
\[S(v)=S(v^{(0)})\le S(v^{(1)}) \le \cdots \le S(v^{(k)}) = S(\pi(w)).\qedhere\]
\end{proof}

We can now prove Conjectures \ref{THEEVENCONJ} and \ref{THEODDCONJ} in the following more general case:

\begin{thm}\label{thm:GeneralCase}
Let $\mathbb W \subset \mathbb P^2$ be a line count configuration of type $c = (c_1, c_2, \ldots, c_t)$ where $c_i \ge i$ for $1 \le i \le t$.  For all integers $r \geq 1$ we have
\begin{enumerate}
\item[(1)] $\alpha((2r-1)\mathbb W) \geq r \alpha(\mathbb W) + (r-1)$;
\item[(2)] $\alpha(2r \mathbb W) \geq r \alpha(\mathbb W) + r$.
\end{enumerate}
\end{thm}

\begin{proof}
If $c$ is GMS, then $H(\mathbb W) = f_c$ and so $\alpha(\mathbb W) = \alpha(diag(c)) = t$.  If $c$ is not GMS then we can add points to $\mathbb W$ to obtain a line count configuration $\mathbb W'$ of type $c' = (c_1', \ldots, c_t')$ so that $\mathbb W \subseteq \mathbb W', c'$ is GMS and $\alpha(diag(c')) = t$.  In this case, $\alpha(\mathbb W) \leq \alpha(\mathbb W') = \alpha(diag(c')) = t$.  In either situation, letting $\ell = 2r-1$ for case (1) and $\ell = 2r$ for case (2), we can find a totally reducing sequence of lines for $\ell \mathbb W$ with associated reduction vector 
$$v = (\ell, \ell, \ldots, \ell) * (c_1, c_2, \ldots, c_t) = (v_1, v_2, \ldots, v_{\ell t})$$
so that $H(\ell \mathbb W) \geq f_v$.  We see that  
$$\alpha(\ell \mathbb W) = \alpha(\Delta H(\ell \mathbb W)) \geq \alpha(f_v) = \alpha(diag(v)) = \ell t +\min(0, v_1-1, v_2-2, \ldots, v_{\ell t}-\ell t).$$
So, for case (1), it suffices to show that $\alpha(diag(v)) \geq rt+r-1$ or, equivalently, to verify the bound $S(v) \leq (r-1)(t-1)$.
For case (2), it suffices to show that $\alpha(diag(v)) \geq rt + r$ or, equivalently $S(v)\leq r(t-1)$.

Note that $v_i \ge w_i$ for $1\le i \le \ell t$, where $w=(\ell,\ell,\ldots,\ell)\circ(1,2,\ldots,t)$.  By Lemma~\ref{genprelemma2}, $S(v)\le S(d)$, where $d=\pi(w)$.  The result then follows by Theorem~\ref{themainthm}.
\end{proof}

\section{Future Directions}\label{future}

Theorem~\ref{thm:GeneralCase} proves Conjectures~\ref{THEEVENCONJ} and \ref{THEODDCONJ} for a large number of types $c$ of line count configurations by showing Claim~\ref{THECLAIM}.  However, it is not possible to prove the bounds on $S(d)$ in Claim~\ref{THECLAIM} for all values of $\ell$ and $t$.  For instance, when $\ell=4$ and $c=(1,1,2,2,3)$, we have $S(d)=10>8=\frac{1}{2}\ell (t-1)$.  It is an interesting question to determine for which $\ell$ and $c$ the bounds in Claim~\ref{THECLAIM} hold.

As shown in the proof of Theorem~\ref{thm:GeneralCase}, if $c$ and $c'$ are two $t$-dimensional non-decreasing integer vectors where $1\le c_i\le c'_i$ for all $1\le i \le t$, then $S((\ell,\ldots,\ell)*c')\le S((\ell,\ldots,\ell)*c)$.  Thus it is natural to ask for the maximal vectors $c$ under this order where $S((\ell,\ldots,\ell)*c)\le \frac{1}{2}(\ell-1)(t-1)$ fails to hold.  The results of Section~\ref{keycase} show that such vectors must be below $(1,2,\ldots,t)$.  However, we believe that when $\ell$ and $t$ are both large, the maximal vectors for which the bound fails are quite far below $(1,2,\ldots,t)$.  For instance, when $\ell=9$ and $c=(1,1,1,2,4,6,7,8,9)$, then $S(d)=31\le 32 = \frac{1}{2}(\ell-1)(t-1)$.

There are counter-examples to the conjectures of Harbourne and Huneke when working over a field of positive characteristic (see \cite{refHS} and the references within).  However, despite the limitations of using the bounds on $S(d)$ from Claim~\ref{THECLAIM}, we still believe that Conjectures~\ref{THEEVENCONJ} and \ref{THEODDCONJ} hold when working over a field of characteristic 0.  In this case, the conjectures remain open for many families of points in addition to line count configurations in $\mathbb{P}^2$ of type $(c_1,\ldots,c_t)$ where not all $c_i\ge i$.  Specifically, the conjectures are open in $\mathbb P ^n$ where $n>2$.  It seems that further progress will require the development of tools similar to the bounds of \cite{refCHT} that work in more general settings.


\end{document}